\documentclass[a4paper]{article}
\usepackage[utf8]{inputenc}
\usepackage{fullpage}
\usepackage{color}
\usepackage{authblk}
\usepackage{latexsym}
\usepackage{amscd}
\usepackage{epsfig}
\usepackage{cancel}
\usepackage{todonotes}
\usepackage{float}
\usepackage{amsmath, amssymb, amsthm, amsfonts, amsgen,mathrsfs} 
\usepackage{stmaryrd}
\usepackage{epstopdf}
\usepackage{caption}
\usepackage{subcaption}
\usepackage{bm}
\usepackage{amsmath}
\usepackage{amssymb}
\usepackage{hyperref}
\usepackage{accents}
\usepackage{latexsym}
\usepackage{amsthm}
\usepackage{fullpage}
\usepackage{graphicx}
\usepackage{color}
\usepackage{booktabs}
\usepackage{amsmath,amsthm,amssymb,amsfonts,hyperref,mathtools,enumitem,bm,ulem}
\newcommand{\R}{\mathbb{R}}
\newcommand{\testfunctions}{\mathbf F}
\newtheorem{theorem}{Theorem}[section]
\newtheorem{lemma}[theorem]{Lemma}
\newtheorem{proposition}[theorem]{Proposition}
\newtheorem{corollary}[theorem]{Corollary}
\theoremstyle{definition}
\newtheorem{definition}[theorem]{Definition}
\newtheorem{example}[theorem]{Example}

\theoremstyle{remark}
\newtheorem{remark}[theorem]{Remark}

 \DeclareMathOperator{\conv}{conv}

\newcommand{\A}{A}
\newcommand{\B}{B}
\newcommand{\Cc}{C}
\newcommand{\OM}{\mathcal{O}^{\mathrm{aff}}}
\newcommand{\OMold}{\mathcal{O}}

\newcommand{\N}{{\rm I\! N}}
\newcommand{\Maff}{M_{\mathrm{r}}^{\mathrm{aff}}}

\newcommand{\Moc}{M_{\mathrm{r}}}
\newcommand{\Mc}{M}
\newcommand{\Munder}{M_{\mathrm{under}}}

\newcommand{\add}{\mathcal{ A}}
\renewcommand{\O}{\mathscr{D}}

\newcommand{\f}{\tilde f}

\newcommand{\cs}{{\mathcal{ S}^1}}

\newcommand{\be}{\begin{eqnarray}}
\newcommand{\ee}{\end{eqnarray}}

\newcommand{\ff}{f}
\newcommand{\COV}{calculus of variations}

\definecolor{dkgreen}{rgb}{0,0.4,0}
\definecolor{dkred}{rgb}{0.8,0.0,0}

\begin{document}
\title{\bf Occupation measure relaxations in  variational problems: the role of convexity}

\author[$1,2$]{ Didier Henrion}
\author[$1,2$]{Milan Korda}
\author[$3,4$]{Martin Kru\v{z}\'{i}k}
\author[$1$]{Rodolfo R\'{\i}os-Zertuche}

\affil[$1$]{\small LAAS-CNRS, Universit\'e de Toulouse, France} 
\affil[$2$]{\small Faculty of Electrical Engineering, Czech Technical University in Prague, Czechia}
\affil[$3$]{\small Czech Academy of Sciences, Institute of Information Theory and Automation,  Prague, Czechia}
\affil[$4$]{\small Faculty of Civil Engineering, Czech Technical University in Prague, Czechia}





\date{\today}

\maketitle

\begin{abstract}
    This work addresses the occupation measure relaxation of calculus of variations problems, which is an infinite-dimensional linear programming relaxation  amenable to numerical approximation by a hierarchy of  semidefinite  optimization problems. We address the problem of equivalence of this relaxation to the original problem. Our main result provides sufficient conditions for this equivalence. These conditions, revolving around the convexity of the data, are simple and apply in very general settings that may be of arbitrary dimensions and may include pointwise and integral constraints, thereby considerably strengthening the existing results. Our conditions are also extended to optimal control problems. In addition, we demonstrate how these results can be applied in non-convex settings, showing that the occupation measure relaxation is at least as strong as the convexification using the convex envelope; in doing so, we prove that a certain weakening of the occupation measure relaxation is equivalent to the convex envelope. This opens the way to application of the occupation measure relaxation in situations where the convex envelope relaxation is known to be equivalent to the original problem, which includes problems in magnetism and elasticity.

\end{abstract}

\section{Introduction}

The moment-sum-of-squares (SOS) hierarchy is a mathematical technology that consists of two steps: 1) formulating a nonlinear  problem (of calculus of variations, optimization or control) as a linear convex optimization problem on the cone of nonnegative measures; 2) solving approximately the infinite-dimensional linear problem on measures by a hierarchy of finite-dimensional convex  (typically semidefinite) optimization problems of increasing size. This hierarchy builds on the duality between the cone of moments and the cone of nonnegative polynomials, and its convergence relies on SOS representations of nonnegative polynomials -- see \cite{hkl20} for a recent overview.  {For calculus of variations and optimal control problems, the linear problem on measures in step 1 is referred to as the {\it occupation measure relaxation}.}

Whereas the convergence analysis of the hierarchy of finite-dimensional convex optimization problems to the solution of the infinite-dimensional linear problem on measures (in step 2) is now well understood,  an essential remaining difficulty  consists of ensuring that there is {\it no relaxation gap} when formulating (in step 1)  the original nonlinear  problem as a linear problem on measures. Indeed, if there is a relaxation gap, {it is currently not understood in full generality how }the approximate solutions obtained by the hierarchy are related to the solutions to the original problem, which is  undesirable. For optimal control or regions of attraction of nonlinear ordinary differential equations, the absence of a relaxation gap was ensured under convexity assumptions \cite{v93}, see \cite[Theorem 2.3]{lhpt08} and \cite[Assumption I]{gq09}, or also \cite[Assumption 2]{hk14}. For scalar hyperbolic conservation laws,  the absence of a relaxation gap was ensured by introducing appropriate entropy inequalities \cite{mwhl20}. See \cite[Remark 6]{khl21} for a discussion on the question of the relaxation gap in the general case of polynomial partial differential equations (PDEs). 

Any attempt to apply the moment-SOS hierarchy to solve nonlinear calculus of variations (\COV{}) problems \cite{d04} is invariably faced with this key relaxation gap issue. In \cite{cbfg21}, no relaxation gap was proved for three very specific \COV{} problems with an integrand which is quadratic convex both in the function and its gradient. Similarly, there is no relaxation gap when the integrand is separably convex in the function and the gradient \cite{ft22}. In \cite{kr22} the absence of a relaxation gap was proved for \COV{} problems with PDE constraints under convexity assumption in the gradient when the dimension of the codomain is 1, with the help of techniques from geometric measure theory. This reference also provides a counterexample with a positive relaxation gap when the codomain has dimension greater than one, despite the Lagrangian density being convex in the gradient. 

The main result of our paper, Theorem \ref{thm:main}, provides a simple proof of the absence of a relaxation gap for \COV{} in a very general dimension-independent setting for problems  with the Lagrangian and constraints that are jointly convex in the function and its gradient, thereby considerably strengthening the existing results. Concretely, on the one hand, the result eliminates the need for a separability of the Lagrangian in~\cite{ft22} and allows for a very broad class of constraints to be considered. On the other hand, the result eliminates the dimensionality dependence of~\cite{kr22}, at the price of stronger convexity assumptions. 

Similarly to~\cite{ft22}, our proof technique works for a large class of measures that satisfy a Liouville condition for test functions that are affine in codomain variables. We analyze this weaker relaxation with affine rather than nonlinear test functions further and prove in Theorem \ref{thm:convexificationb} that when it is applied to a nonconvex problem its infimum is equal to the infimum of the original problem convexified by taking the classical convex envelope. Therefore, for nonconvex problems, the value of the occupation measure relaxation with nonlinear test functions is sandwiched between the infimum of the original problem and the infimum of the convexification using the convex envelope. Consequently, the occupation measure relaxation is at least as strong as the convexification using the convex envelope. In particular, whenever one can prove the absence of a relaxation gap for the latter, it holds for the former. 

The convexity assumptions of our result hold for some problems in continuum physics, namely for static micromagnetics \cite{desim}, which we explore in detail in Section \ref{sec:magnetics}, and for some models of elasticity \cite{kohn, raoult}.
%
%
%

This opens the way to application of the occupation measure relaxation in situations where the convex envelope relaxation is known to be equivalent to the original problem. For these problems, the use of Young measures (also called parametrized measures) is classical, see e.g. \cite{y69,roubicek,pedregalenv}. As already explained in \cite{hkw19}, occupation measures can be seen as an extension of Young measures allowing the application of numerical methods based on the moment-SOS hierarchy.
    
The paper is organized as follows. In Section \ref{sec:statement} we state the \COV{} problem to be solved, and we introduce the occupation measures we are using to reformulate it as a linear optimization problem.
Section \ref{sec:convex} contains our main result Theorem \ref{thm:main} about the absence of a relaxation gap between the \COV{} problem and its linear formulation, under convexity assumptions. Non-convex problems are the subject of Section \ref{sec:nonconvex} where we show how our linear formulation relates with the classical convex envelope relaxation. Section \ref{sec:optimalcontrol} explains how our \COV{} results can be extended to optimal control problems. Finally, \COV{} applications in micromagnetics are described in Section \ref{sec:magnetics}.

\section{Calculus of variations and linear formulations with occupation measures}\label{sec:statement}

Let $\Omega \subseteq \R^n$ be a connected, bounded  domain with piecewise $C^1$ boundary $\partial\Omega$,
let $Y \subset \R^m$ and $Z \subset \R^{m\times n}$. Denote by $\sigma$ the $(n-1)$-dimensional Hausdorff measure on the boundary $\partial \Omega$. Let us denote the transpose of a matrix $A$ by $A^\top$.

Consider the calculus of variations 
problem
\begin{alignat}{2}\label{opt:classical}
 \Mc \coloneqq \inf_{y \in W^{1,p}(\Omega ; Y)}&\quad&& \int_\Omega L(x,y(x),Dy(x))\,dx\\
 \text{subject to}&&&
 \A_i(x,y(x),Dy(x)) = 0,\quad \B_i(x,y(x),Dy(x)) \le 0 \quad \text{a.e.\ } x\in \Omega ,\;i\in I,\label{eq:pde}\\
&&&\label{eq:boundary}
\A_{\partial,i}(x,y(x)) = 0,\quad \B_{\partial,i}(x,y(x)) \le 0 \qquad\text{$\sigma$-a.e. } x\in\partial\Omega,\; i\in I,\\
&&&\label{eq:integralconstraints}
\int_\Omega \Cc_i(x,y(x),Dy(x))\,dx\leq 0,\quad \int_{\partial\Omega} \Cc_{\partial,i}(x,y(x))\,d\sigma(x)\leq 0,\quad i\in I,
\end{alignat}
{where $I$ is a possibly uncountable index set and  $W^{1,p}(\Omega;Y)$ denotes the Sobolev space of measurable functions $\phi\colon \Omega\to Y$ that are weakly differentiable and such that both $\phi$ and its weak derivative $D\phi$ are in $L^p$. Considering constraints indexed by a possibly uncountable index set provides a significant modeling freedom, allowing one, for example, to model higher-order PDE constraints in the weak form as shown in Example \ref{ex:constraints}. }

\begin{example}[Plateau's problem]
Let $\Omega$ be the unit disc in $\R^2$, $Y=\R^3$, and $Z=\R^{2\times 3}$.
 Let $\gamma\colon[0,2\pi]\to \R^3$ be a continuous, non-self-intersecting loop, $\gamma(0)=\gamma(2\pi)$.
 let $L(x,y,Dy)=\left\|\frac{\partial y}{\partial x_1}\times \frac{\partial y}{\partial x_2}\right\|$, that is, $L$ is the area of the parallelogram spanned by the columns of $Dy$, known as the element of surface area. Let $A_{\partial,1}(x,y)$ be the distance from $y$ to the image $\gamma([0,2\pi])$. Let all other constraint functions $A_i,B_i,C_i,A_{\partial,i},B_{\partial_i},C_{\partial,i}$ be zero. In this setting, $M$ is the area of the minimal surface with boundary $\gamma([0,2\pi])$.
\end{example}

\begin{example}[Geodesics of a differential inclusion]
    Let $\Omega$ be the interval $[0,1]$ and let $V$ be a connected, open subset of $Y=\R^n$, and $Z=\R^n$. Fix two points $y_0$ and $y_1$ in $V$, as well as a Riemannian metric $g\colon \R^n\times \R^n\to\R_{\geq 0}$. Let $B_1\colon Y\times Z\to\R$ be a continuous function such that the sets $X(y)=\{z\in Z:B(y,z)\leq 0\}$ are nonempty for each $y\in Y$. Let $L(x,y,z)=\sqrt{g(z,z)}$, let $A_1(x,y,z)$ be the distance from $y$ to $V$, and $A_{\partial,1}(x,y,z)$ be the distance from $y$ to the set $\{y_0,y_1\}$. Then $M$ is the minimal $g$-distance from $y_0$ to $y_1$ for curves joining these points and respecting the constraint $\gamma(t)\in V$ and the differential inclusion $\gamma'(t)\in X(\gamma(t))$, and the minimizers are the geodesics joining these two points and respecting the differential inclusion. When $X(y)$ is a linear subspace of $Z$ for each $y$, this setting is known as  sub-Riemannian geometry.
\end{example}




\begin{example}[Higher-order PDEs]\label{ex:constraints}
Considering an uncountable index set in~\eqref{opt:classical} permits higher-order PDEs to be considered as constraints in~\eqref{opt:classical} through their weak reformulation. We demonstrate this on the Laplace equation.
 In the $m=1$ case, to model a {Laplace equation in the weak form}, we can use the constraint \[\int_{\Omega} \nabla\phi(x)^\top\nabla y(x)\,dx=0,\qquad\phi\in C^1_0(\Omega).\]
 Observe that if $y$ were a smooth function, integrating by parts we would obtain $\int \phi(x)\,\Delta y(x)\,dx=0$ for all $\phi\in C^1_0(\Omega)$, whence the above conditions indeed give a weak analogue of the Laplace equation $\Delta y=0$.
 In order to model it as in~(\ref{eq:integralconstraints}), we set $I$ to be $C^1_0(\Omega)$, and for each $\phi\in I=C^1_0(\Omega)$ we set $\Cc_\phi(x,y,z)=\nabla\phi(x)^\top z$. Observe that the equality constraint results from the inequalities for $\phi$ and $-\phi$, which are both contained in $C^1_0(\Omega)$.
\end{example}


To reformulate \COV{} problem \eqref{opt:classical} as a linear optimization problem, we now introduce two classes of measures. The first ones, the \textit{affinely relaxed occupation measures}, form a larger class that contains the second ones, the \textit{relaxed occupation measures}. Our main no-gap results will hold for the first class of measures (see Theorems \ref{thm:main} and \ref{thm:optimalcontrol}), and hence will also hold for the second one (see Corollary \ref{coro:unrelax} and Remark \ref{rmk:affineOC}). The second class is more commonly found in the literature, and we chose to introduce the first class in order to present our results in greater generality.

\begin{definition}[Affinely relaxed occupation measures]\label{def:M}
 Let $p\in [1,+\infty]$.
 Let $\OM_p$ be the set of pairs $(\mu,\mu_\partial)$ consisting of  positive Radon measures on $\overline\Omega\times Y\times Z$ respectively $\partial\Omega\times Y$ satisfying: 
 \begin{itemize} 
 \item The total mass of $\mu$ equals the Lebesgue measure of $\Omega$, that is, \begin{equation}\label{eq:measureomega}
 \mu(\overline\Omega\times Y\times Z)=|\Omega|.
\end{equation}
 \item The sets
 $\operatorname{supp}\mu$ and $\operatorname{supp}\mu_\partial$ are compact if $p=+\infty$, or
\begin{equation}\label{eq:Lp}
\int_{\Omega\times Y\times Z}\|y\|^p+\sum_i\|z_i\|^p\,d\mu(x,y,z)<+\infty\quad\text{and}\quad \int_{\partial\Omega\times Y}\|y\|^p\,d\mu_\partial(x,y)<+\infty
\end{equation}
if $p\in[1,+\infty)$. Here, $z=(z_1,\dots,z_n)$, $z_i\in \R^m$.

\item For all $\phi_1\in C^\infty(\Omega;\R)$ and all $\phi_2\in C^\infty(\Omega;\R^m)$ it holds
\begin{multline}\label{eq:boundarymeasure}
  \int_{\Omega\times Y\times Z}\left[\frac{\partial\phi_1}{\partial x}(x)+y^\top\frac{\partial\phi_2}{\partial x}(x)+\phi_2(x)^\top z\right]d\mu(x,y,z)\\
  =\int_{\partial\Omega\times Y} \left(\phi_1(x)+y^\top\phi_2(x)\right)\mathbf n(x)\,d\mu_\partial(x,y).
 \end{multline}
 Here $\mathbf n$ denotes the exterior unit (row) vector normal to the boundary $\partial\Omega$.

\end{itemize}
\end{definition}

\begin{remark}\label{rmk:liouville} 
Observe that \eqref{eq:boundarymeasure} amounts to the Liouville equation 
\[\int_{\Omega\times Y\times Z}\frac{\partial\phi}{\partial x}+\frac{\partial\phi}{\partial y}z\,d\mu=\int_{\partial\Omega\times Y}\phi\,\mathbf n\,d\mu_\partial  \]
for a test function $\phi\in C^\infty(\Omega\times Y;\R)$ that is affine in $y$, that is, of the form $\phi(x,y)=\phi_1(x)+y^\top\phi_2(x)$.
\end{remark}

\begin{remark}
The assumption of $\Omega$ being connected can be dropped at the price of replacing~(\ref{eq:measureomega}) by the slightly stronger requirement that its $x$-marginal  is equal to the Lebesgue measure on $\Omega$. When $\Omega$ is connected and its mass is equal to the volume of $\Omega$ (condition~\eqref{eq:measureomega}), its $x$-marginal is automatically equal to the Lebesgue measure on $\Omega$ due to~\eqref{eq:boundarymeasure} and Lemma~\ref{lem:lebesgueproj} in the appendix.
\end{remark}

\begin{definition}[Relaxed occupation measures]\label{def:M2}
Let $p\in[1,+\infty]$. Let $\OMold_p$ be the set of pairs $(\mu,\mu_\partial)\in \OM_p$ that additionally satisfy:
\begin{itemize}
\item For all $\phi \in \testfunctions_p$, it holds
\begin{equation}\label{eq:boundarymeasure2}
  \int_{\Omega\times Y\times Z}\frac{\partial\phi}{\partial x}(x,y)+\frac{\partial \phi}{\partial y}(x,y)z\,d\mu(x,y,z)
  =\int_{\partial\Omega\times Y} \phi(x,y)\mathbf n(x)\,d\mu_\partial(x,y),
 \end{equation}
where $\mathbf n$ denotes the exterior unit vector normal to the boundary $\partial\Omega$,
\begin{multline}
 \notag\testfunctions_p=\Big\{\phi\in C^\infty(\Omega\times Y):\exists\, c>0\;\text{such that}\;\left\|\frac{\partial\phi}{\partial x}(x,y)\right\|\leq c(1+\|y\|^p),\\\left\|\frac{\partial\phi}{\partial y}(x,y)\right\|\leq c(1+\|y\|^{p-1}),\;
  |\phi(x,y)|\leq c (1+\|y\|^p)  \Big\}, \quad p\in [1,\infty)
 \end{multline}
 and $\testfunctions_\infty = C^\infty(\Omega\times Y)$.
 By Lemma \ref{lem:integrability}, the integrals in~\eqref{eq:boundarymeasure2} are well defined.
\end{itemize}
\end{definition}

Consider the  relaxations with affinely relaxed occupation measures (Definition \ref{def:M}),
\begin{alignat}{2}\label{opt:relaxed}
\Maff\coloneqq\inf_{(\mu,\mu_\partial)\in \OM_p}&\quad&& \int_{\Omega\times Y\times Z} L(x,y,z)\,d\mu(x,y,z) \\
\nonumber\textrm{subject to}&& &    \operatorname{supp}\mu \subseteq \{(x,y,z)\in\Omega\times Y\times Z\mid  \A_i(x,y,z) = 0,\;B_i(x,y,z)\leq 0,\;i\in I \}, \\
   \nonumber&&& \operatorname{supp}\mu_\partial \subseteq \{(x,y)\in\partial\Omega\times Y\mid \A_{\partial,i}(x,y)=0,
   \;\B_{\partial,i}(x,y) \leq 0,\;i\in I\}, \\
   \nonumber&&&\int_{\Omega\times Y\times Z}\Cc_i(x,y,z)\,d\mu(x,y,z)\leq 0,\quad i\in I,\\
   \nonumber&&&\int_{\partial\Omega\times Y}\Cc_{\partial,i}(x,y)\,d\mu_\partial(x,y)\leq 0,\quad i\in I,
\end{alignat}
and with relaxed occupation measures (Definition \ref{def:M2}),
\begin{alignat}{2}\label{opt:relaxedold}
\Moc\coloneqq\inf_{(\mu,\mu_\partial)\in \OMold_p}&\quad&& \int_{\Omega\times Y\times Z} L(x,y,z)\,d\mu(x,y,z) \\
\nonumber\textrm{subject to}&& & \text{[same constraints as in \eqref{opt:relaxed}].}
\end{alignat}
Remark that the sole difference between \eqref{opt:relaxed} and \eqref{opt:relaxedold} is the set over which the infimum is taken, namely, $\OM_p$ and $\OMold_p$, respectively.

Observe also that problems \eqref{opt:relaxed} and \eqref{opt:relaxedold} are {\it linear} optimization problems in the unknown occupation measures $(\mu,\mu_{\partial})$, whereas the original \COV{} problem \eqref{opt:classical} can be {\it nonlinear} in the unknown function $y$.

\section{Convex case -- main result}\label{sec:convex}

Our main result states that under convexity assumption there is no relaxation gap between the original \COV{} problem and its linear reformulation with affinely relaxed occupation measure.

\begin{theorem}\label{thm:main}
Suppose that for every $i\in I $ and $x \in \Omega$ the following holds: $L$, $\B_i$ and $\Cc_i$ are convex in $(y,z)$, and $\A_i$ is affine in $(y,z)$. Suppose also that for every $i\in I$ and $x \in \partial\Omega$ the following holds: $\A_{\partial,i}$ is affine in $y$, and $\B_{\partial,i}$ and $\Cc_{\partial,i}$ are convex in $y$. Then, if $Y$ and $Z$ are closed and convex, we have 
\[\Maff =\Mc.\]
Recall that these are defined in \eqref{opt:classical} and \eqref{opt:relaxed}.
\end{theorem}
\begin{proof}
Clearly we have $\Maff \le\Mc$ because every function $y$ induces measures $(\mu,\mu_\partial)\in \OM_p$ by letting
\begin{gather*}\int_{\Omega\times Y\times Z}\phi(x,y,z)\,d\mu(x,y,z)=\int_{\Omega}\phi(x,y(x),Dy(x))\,dx,\quad \phi\in C^1(\Omega\times Y\times Z),\\
\int_{\partial\Omega\times Y}\phi(x,y)\,d\mu_\partial(x,y)=\int_{\partial\Omega}\phi(x,y(x))\,d\sigma(x),\quad \phi\in C^1(\Omega\times Y).
\end{gather*}
The harder part is to prove that $\Maff \ge\Mc$. To this end, let $(\mu,\mu_\partial)$ be feasible in~(\ref{opt:relaxed}). By Lemma \ref{lem:lebesgueproj}, the projection of $\mu$ onto $\Omega$ is Lebesgue measure $dx$. Disintegrate $\mu$ such that
\begin{equation}\label{eq:disintegration}
\mu(x,y,z) = \nu(y,z\mid x)\,dx,
\end{equation}
where $\nu$ is the conditional distribution of $(y,z)$ given $x$; in particular, for every fixed $x$, $\nu(\cdot,\cdot \mid x)$ is a probability measure. 
Define
\begin{equation}\label{eq:defyz}
y(x) := \int_{Y\times Z} y\,d\nu(y,z\mid x),\quad z(x) := \int_{Y\times Z} z\,d\nu(y,z\mid x)
\end{equation}
to be the conditional expectation (or centroid) of $\mu$ in the $y$ and $z$ variables given $x$. We note that $y(x) \in \R^m$ and $z(x)\in \R^{m\times n}$. By Jensen's inequality we have
\[
\int_{\Omega\times Y\times Z} L \,d\mu = \int_{\Omega} \int_{Y\times Z} L(x,y,z)\,d\nu(y,z\mid x)\,dx \ge \int_\Omega L(x,y(x),z(x))\,dx.
\]
Now we shall prove that $y(\cdot)$ is weakly differentiable and $Dy(x) = z(x)$.
To prove weak differentiability of $y(\cdot)$, take 
$\phi_1=0$ and $\phi_2=\psi \in C_0^\infty(\Omega;\R^m)$, 
and use it as a test function 
in~(\ref{eq:boundarymeasure}). This gives
\[
\int_{\Omega\times Y\times Z}y^\top\frac{\partial\psi}{\partial x}+\psi(x)^\top z\,d\mu(x,y,z) = 0.
\]
Using the disintegration of $\mu$ we obtain
\[
\int_{\Omega} \int_{Y\times Z} y^\top \frac{\partial\psi}{\partial x}+\psi(x)^\top z\,d\nu(y,z\mid x) dx = 0.
\]
Since the integrand is linear in $(y,z)$, we obtain, using the definitions of $y(\cdot)$ and $z(\cdot)$,
\[
\int_{\Omega} y(x)^\top\frac{\partial \psi}{\partial x} + \psi(x)^\top z(x)\,dx = 0,\quad \forall \psi \in C_0^\infty(\Omega).
\]
This is nothing but the definition of weak differentiability of $y(\cdot)$ with weak derivative $z(\cdot)$. 
If $p=+\infty$, $\operatorname{supp}\mu$ is bounded, and then $z(\cdot)$ is $L_\infty$, so it follows that $y \in W^{1,\infty}$.
If $p\in [1,+\infty)$, then $z(\cdot)$ is in $L_p$; indeed, by Jensen's inequality applied to the convex function $\|\cdot\|^p$, we have, writing $z=(z_1,\dots,z_m)$,
\begin{align*}&\int_{\Omega\times Y\times Z}\|z_i(x)\|^p\,dx=\int_\Omega\left\|\int_{Y\times Z}z_i\,d\nu(y,z\mid x)\right\|^pdx\\
&\leq \int_\Omega\int_{Y\times Z}\left\|z_i\right\|^p\,d\nu(y,z\mid x)dx=\int_{\Omega\times Y\times Z}\|z_i\|^p\,d\mu(x,y,z)<+\infty,
 \end{align*}
 by \eqref{eq:Lp}. A similar calculation shows that $y(\cdot)$ is in $L_p$ as well.
 It follows that $y$ is in $W^{1,p}$.

We also need to verify the constraints (\ref{eq:pde})--
(\ref{eq:integralconstraints}). This follows from the support constraints on $\mu$ and $\mu_\partial$ and the fact that $\A_i$ and $\A_{\partial,i}$ are affine and $\B_i$, $B_{\partial,i}$, $C_i$ and $C_{\partial,i}$ are convex in $(y,z)$. Indeed, the support constraints imply
\[
\int \psi(x) \A_i(x,y,z)\,d\mu(x,y,z) = 0 \quad \forall \psi \in C^0(\Omega).
\]
Using the disintegration \eqref{eq:disintegration}, and definition of $y(\cdot)$, $z(\cdot)$ and the fact that $\A_i$ is affine, we obtain
\[
\int \psi(x) \A_i(x,y(x),z(x))\,dx = 0 \quad \forall \psi \in C^0(\Omega)
\]
and hence the constraint $\A_i(x,y(x),z(x)) = 0$ is satisfied a.e.\ in $\Omega$, $i\in I$. For the inequality constraint $\B_i(x,y(x),Dy(x)) \le 0$, we observe that for any nonnegative $\psi \in C^0(\Omega)$, we have
\[
0 \ge \int \psi(x)\B_i(x,y,z) \,d\mu = \int_\Omega \int_{Y\times Z} \psi(x)\B_i(x,y,z) \,d\nu(y,z\mid x)\,dx \ge \int_\Omega \psi(x)\B_i(x,y(x),z(x))\,dx ,
\]
where the first inequality follows from the support constraints on $\mu$ and the last one from Jensen's inequality. Since $z(x) = Dy(x)$, this implies that $\B_i(x,y(x),Dy(x)) \le 0 $ a.e.\ in $\Omega$. Another application of Jensen's inequality gives, for all $i\in I$,
\[\int_{\Omega}\Cc_i(x,y(x),z(x))\,dx\leq \int_{\Omega}\int_{Y\times Z}\Cc_i(x,y,z)\,d\nu(y,z\mid x)\,dx=\int_{\Omega\times Y\times Z}\Cc_i(x,y,z)\,d\mu(x,y,z)\leq 0.\]

Let us now verify the boundary condition. To this end, disintegrate $\mu_\partial$ as 
\begin{equation}\label{eq:disintegrationbdry}
\mu_\partial = \nu_\partial(y\mid x)d\sigma(x)
\end{equation}
with $\sigma$ being the Hausdorff measure on $\partial \Omega$.  Define
\[
y_\partial(x) = \int y\,d\nu_\partial(y\mid x).
\]
We shall show that $y_\partial(\cdot)$ is a Sobolev trace of $y(\cdot)$ (see~\cite[Definition, 4.5]{EvansGariepy}). To this end, take $\phi_1=0$ and $\phi_2=\psi \in C^\infty(\Omega;\mathbb{R}^m)$, so that plugging 
in~(\ref{eq:boundarymeasure}) and using the fact that $z(x) = Dy(x)$ almost everywhere in $\Omega$ gives
\[
\int_{\Omega} y(x)^\top\frac{\partial \psi}{\partial x} + \psi(x)Dy(x)\,dx = \int_{\partial\Omega} \psi(x)^\top y_\partial(x) \mathbf{n}(x)\,d\sigma(x),\quad \forall \psi \in C^\infty(\mathbb{R}^n).
\]
This is the Stokes theorem for the Sobolev functions which shows that $y_\partial(\cdot)$ is indeed a Sobolev trace of $y(\cdot)$~\cite[Theorem~4.6]{EvansGariepy}. Hence, in order to show that $\A_{\partial,i}(x,y(x)) = 0$ on $\partial\Omega$ it suffices to prove it for $y_\partial$. The second support constraint of~(\ref{opt:relaxed}) implies that
\[
\int_{\Omega\times Y} \phi(x) \A_{\partial,i}(x,y)\,d\mu_\partial(x,y) = 0 \quad \forall \phi \in C^0(\partial\Omega).
\]
Using the disintegration \eqref{eq:disintegrationbdry}, the definition of $y_\partial$ and the fact that $\A_\partial$ is affine in $y$, it follows that
\[
\int_{\Omega\times Y} \phi(x) \A_{\partial,i}(x,y_\partial(x))\,dx= 0 \quad \forall \phi \in C^0(\partial\Omega).
\]
This implies that $y_\partial(\cdot)$ and hence $y(\cdot)$ satisfy the boundary condition $\sigma$ a.e.\ in $\partial\Omega$. For $p = +\infty$, $y(\cdot)$ is Lipschitz and hence $y(\cdot)$ satisfies it in fact everywhere in $\partial\Omega$. The boundary inequality constraint $\B_{\partial,i}(x,y(x)) \le 0$ follows by Jensen's inequality the same way as for $\B_i(x,y(x),Dy(x)) \le 0$, and similarly for the integral boundary constraint \eqref{eq:integralconstraints}.

It remains to verify that $y(x)\in Y$ for all $x\in\Omega$. This, however, follows by the fact that $Y$ is closed and convex and the last support constraint of~(\ref{opt:relaxed}).
\end{proof}

\begin{corollary}\label{coro:unrelax}
Under the assumptions of Theorem~\ref{thm:main}, we have $\Moc =\Maff=\Mc$.
\end{corollary}
\begin{proof}
 Since $\OMold_p\subseteq\OM_p$ and $\OMold_p$ contains the pairs of measures induced by functions in $W^{1,p}$, we have that
 $\Maff\leq \Moc\leq \Mc$. Moreover, by Theorem \ref{thm:main} we know that $\Mc=\Maff$, so $\Moc$ must be equal to them as well.
\end{proof}

\section{Nonconvex case}\label{sec:nonconvex}

{We now turn to the case when the problem is not convex and show that the affinely relaxed occupation measures provide a lower bound at least as good as the the convex envelope relaxation and, in the absence of integral constraints, the two lower bounds coincide.

\subsection{Convex underestimators}
First, we make an observation comparing the value of the occupation measure relaxation to the value of the original problem with data replaced by its convex underestimators. Concretely, suppose that there exist convex sets $\hat Y \supset Y$ and  $\hat Z\supset Z$, functions $\hat L,\hat B_i,\hat C_i\colon\Omega\times\hat Y\times\hat Z \to\R$ such that for all $(x,y,z) \in \Omega\times Y\times Z$ satisfying $A_i(x,y,z) = 0$, $B_i(x,y,z) \le 0$, it holds $\hat L(x,y,z) \le L(x,y,z)$, $\hat B_i(x,y,z) \le B_i(x,y,z)$, $\hat C_i(x,y,z) \le C_i(x,y,z)$. Suppose also that there exist  functions $\hat L_\partial,\hat B_{\partial,i},\hat C_{\partial,i}\colon\Omega\times\hat Y\to\R$ such that for all $(x,y) \in \partial\Omega\times Y$ satisfying $A_{\partial,i}(x,y) = 0$, $B_{\partial,i}(x,y) \le 0$, it holds $\hat L_\partial(x,y) \le L_\partial(x,y)$, $\hat B_\partial(x,y) \le B_{\partial,i}(x,y)$ and $\hat C_{\partial,i}(x,y) \le C_{\partial,i}(x,y)$. Consider the problem
\begin{alignat}{2}
 \Munder =
 \inf_{y \in W^{1,p}(\Omega,\hat Y)} 
  &\quad&&
 \int_{\Omega}\hat L(x,y(x),Dy(x))dx
  +
 \int_{\partial\Omega} \hat L_\partial(x,y(x))\,d\sigma(x) \label{opt:convexified}\\
  \text{subject to}&&& 
   \A_i(x,y(x),Dy(x))=0,\quad \hat \B_i(x,y(x),Dy(x))\leq 0,\quad x\in\Omega,\;i\in I, \nonumber\\
  &&& \A_{\partial,i}(x,y(x))=0,\quad \hat\B_{\partial,i}(x,y(x))\leq 0,\quad x\in\partial\Omega,\;i\in I, \nonumber\\
  &&& \int_{\Omega}\hat\Cc_i(x,y(x),Dy(x))\,dx\leq 0,\quad \int_{\partial\Omega}\hat\Cc_{\partial,i}(x,y(x))\,dx\leq 0,\quad i\in I. \nonumber
\end{alignat}
We have the following simple corollary, where we for simplicity assume that $A_i$ and $A_{\partial,i}$ are affine, i.e., they do not need further convexification. This could be relaxed by replacing $A_i$ and $A_{\partial,i}$ by $A_i^2$ and $A_{\partial,i}^2$ and considering convex underestimators thereof.
\begin{corollary}\label{coro:nonconvex}
Suppose that $A_i$ is affine in $(y,z)$ for every $x\in \Omega$ and $A_{\partial,i}$ is affine in $y$ for every $x\in \partial\Omega$. It holds $\Munder \le \Maff \le M_{\mathrm{r}} \le\Mc$ (defined, respectively, in \eqref{opt:convexified}, \eqref{opt:relaxed}, \eqref{opt:relaxedold}, and \eqref{opt:classical}). In particular, if  $\Munder = \Mc $, then $\Maff = M_{\mathrm{r}} = \Mc $. 
\end{corollary}
\begin{proof}
Let $\hat M_\mathrm{r}$ denote the infimum of the occupation measure relaxation  applied to the convexified problem~(\ref{opt:convexified}).
Applying Corollary~\ref{coro:unrelax}, we have $\hat M_\mathrm{r} = \Munder $. Given a pair of measures $(\mu,\mu_\partial)$ feasible in~(\ref{opt:relaxed}), we observe that these two measures are feasible in the occupation measure relaxation of the convexified problem since the feasible set of the latter is larger than the feasible set of the original problem. Since $\hat L \le L$, it follows that $\hat M_{\mathrm{r}} \le \Maff$ and as a result $\Munder\le \Maff \le M_{\mathrm{r}} \le \Mc $ since always $\Maff \le M_{\mathrm{r}} \le M$.
\end{proof}

Note that when $\hat \Mc  =  \Mc  = \Maff = \Mc_{\mathrm{r}}$, one has two choices for a relaxation to be solved numerically: either the convexification~(\ref{opt:convexified}) or the occupation measure relaxation~(\ref{opt:relaxed}). Both are convex optimization problems and it depends on particular situation whether one or the other is more advantageous to solve. For example, the occupation measure relaxation may be preferable when no explicit description of the  nonnegative convex underestimators used in~(\ref{opt:convexified}) is available.}

\subsection{Equivalence with convex envelope}

Now we significantly refine the previous observation and prove that in fact $\Maff$ is equal to the infimum of a convexification of the problem based on the convex envelope of the Lagrangian and the constraints. As a result, under the assumptions introduced below, we obtain the following inequalities
\[
\text{Convex envelope} = \text{Affinely relaxed OM} \le \text{Relaxed OM} \le \text{Original},
\]
where OM stands for ``occupation measures.'' To set the stage for proving this, we let $\conv X$ denote the closed convex hull of a set $X \subset \R^n$. Let $p=+\infty$, and let $Y$ and $Z$ be compact sets. For $x\in\Omega$, let $K(x)$ denote the set of points $(y,z)\in Y\times Z$ such that the constraints $A_i(x,y,z)=0$ and $B_i(x,y,z)\leq 0$ are satisfied. We set $C_i = 0$, $C_{i,\partial} = 0$, i.e., we work without integral constraints in this section, {except for Examples  \ref{ex:integralconstraints} and \ref{ex:integralconstraints2}, that will show that we would not be able to obtain the same results in the presence of such constraints.}
 
 Let $\hat L\colon 
 \Omega\times \conv Y\times\conv Z\to\R\cup\{+\infty\}$ be the function whose restriction to $\{x\}\times \operatorname{conv} K(x)$ is the {convex envelope} of $L|_{\{x\}\times K(x)}$; in other words, for $x\in \Omega$,
  \begin{align*}
 \hat L(x,y,z)=\sup\{\varphi(y,z)\mid \;&\text{$\varphi\colon \conv Y\times \conv Z\to\R$ convex, }\\
 & \varphi(y',z')\leq L(x,y',z')\;\forall \,(y',z')\in K(x)\}.
 \end{align*}

   Analogously, let $K_\partial(x)\subseteq Y$ be, for each $x\in \partial\Omega$, the set of points $y\in Y$ such that $A_{\partial,i}(x,y)=0$ and $B_{\partial,i}(x,y)\leq 0$. 
Let $\hat A_i(x)=\conv (A_i|_{\{x\}\times Y\times Z})^{-1}(0)$ and $\hat A_i=\bigcup_{x\in\Omega}\hat A_i(x)$; in other words, $\hat A_i$ is the convexification in the $Y\times Z$ directions of the set $A_i^{-1}(0)$. Similarly, let $\hat A_{\partial,i}(x)=\conv (A_{\partial,i}|_{\{x\}\times Y})^{-1}(0)$ and $\hat A_{\partial,i}=\bigcup_{x\in\partial\Omega}\hat A_{\partial,i}(x)$.
 
Let
\begin{alignat}{2}
 \hat \Mc =&
 \inf_{y \in W^{1,\infty}(\Omega,\conv Y)} 
  &\quad&
 \int_{\Omega}\hat L(x,y(x),Dy(x))dx
 \label{opt:convexified2b}\\
  &\text{subject to}&& 
   (y(x),Dy(x))\in \conv K(x),\quad \text{a.e.\ }x\in \Omega, \nonumber\\
   &&&y(x)\in\conv K_\partial(x),\quad\text{a.e.\ }x\in\partial\Omega. \nonumber
\end{alignat}


The following result shows  that the convexified problem \eqref{opt:convexified2b} has the same infimum as the occupation measure relaxation \eqref{opt:relaxed} with affine test functions applied to the original problem with nonconvex data.
\begin{theorem} \label{thm:convexificationb}
 Assume that $K(x)$ is closed for each $x\in\Omega$ and that the set-valued map $K$ is continuous in the topology induced by the Hausdorff metric. 
 If $L$ is continuous in $(y,z)$ {for every $x \in \Omega$},
 then it holds that
 \[\hat \Mc=\Maff.\]
\end{theorem}
\begin{proof}
 Let us first show that $\Maff\leq \hat \Mc$.
 Let $y\in W^{1,p}(\Omega,\conv Y)$ be viable for the optimization problem $\hat\Mc$. By Choquet's Theorem \cite{lukevnetukavesely}, we know that, for each $x\in\Omega$ where $Dy(x)$ is defined, there is a probability measure $\nu_x$ supported in $K(x)$ such that 
 \begin{equation}
     \label{eq:choquet}
     (y(x),Dy(x))=\int_{K(x)}(y,z)\,d\nu_x(y,z)\quad\text{and}\quad\hat L(x,y(x),Dy(x))=\int_{K(x)}L(x,y,z)\,d\nu_x(y,z).
 \end{equation}

Let us show that the mapping $x\mapsto\nu_x$ can be chosen so that it is measurable. Let $\mathcal M_x$ be the set of measures $\nu$ on $K(x)$ such that $\int_{K(x)}(y,z)\,d\nu(y,z)=(y(x),Dy(x))$. Endow $\mathcal M_x$ with the weak* topology with respect to $C^0(Y\times Z;\R)$. Let 
\[I_L(\nu,x)=\int_{\{x\}\times Y\times Z}L\,d\nu\qquad\text{and}\qquad\phi(x)=\min_{\nu\in \mathcal M_x} I_L(\nu,x)=\hat L(x,y(x),Dy(x)).\] 
With these definitions, $I_L$ is continuous and $\phi$ is measurable. By the Prokhorov theorem, $\mathcal M_x$ is compact and metrizable. The inverse map $\Phi_x=I_L(\cdot,x)^{-1}$, that is, $\Phi_x(r)=\{\nu\in \mathcal M_x:I_L(\nu,x)=r\}$, $r\in\R$, is weakly measurable as well, since every open set $U\subset \mathcal M_x$ can be written as a countable union of compact sets $U=\bigcup_iK_i$, and $\Phi_x^{-1}(U)=\bigcup_i\Phi_x^{-1}(K_i)=\bigcup_i I_L(K_i,x)$ is a countable union of continuous images of compact sets, each of which is compact, so the union is measurable. Then the set-valued map 
\[H(x)=I_L(\cdot,x)^{-1}(\phi(x))=\operatornamewithlimits{arg\,\min}_{\nu\in\mathcal M_x}\int_{\{x\}\times Y\times Z}L\,d\nu \]
is measurable, as it is the composition of two measurable maps. The image $H(x)$ is closed, so we may apply the  Kuratowski--Ryll-Nardzewski Selection Theorem \cite[Th. 18.13]{aliprantisborder} to get an appropriate measurable function $x\mapsto\nu_x\in H(x)$.

Letting $\mu=\nu_x(y,z)\,dx$ we obtain a measure that is viable for $\Maff$; to see that it satisfies \eqref{eq:boundarymeasure} note that since $\nu_x\in\mathcal M_x$, for each $x$ the averages \eqref{eq:defyz} on $Y$ and $Z$ coincide almost everywhere with $y(x)$ and $z(x)=Dy(x)$, respectively, and \eqref{eq:boundarymeasure} is linear in these variables, so the fact that the Liouville equation for affine test functions,
\[ \int_{\Omega}\left[\frac{\partial\phi_1}{\partial x}(x)+y(x)^\top\frac{\partial\phi_2}{\partial x}(x)+\phi_2(x)^\top Dy(x)\right]dx
  =\int_{\partial\Omega} \left(\phi_1(x)+y(x)^\top\phi_2(x)\right)\mathbf n(x)\,d\sigma(x),\]
is verified for $y(\cdot)$ implies that \eqref{eq:boundarymeasure} is verified for $\mu$. The measure $\mu$ also satisfies 
\begin{equation*}
\int_{\Omega\times Y\times Z} L(x,y,z)\,d\mu(x,y,z)=\int_{\Omega}\int_{  Y\times Z} L(x,y,z)\,d\nu_x(y,z)\,dx=\int_\Omega \hat L(x,y(x),Dy(x))\,dx
\end{equation*}
by \eqref{eq:choquet}.
Arguing analogously in the boundary $\partial\Omega$ to obtain a measure $\mu_\partial$ supported in $\bigcup_{x\in\partial \Omega}\{x\}\times K_\partial(x)$ from $y(\cdot)|_{\partial\Omega}$, this implies that $\Maff\leq\hat\Mc$. 

Now we turn to showing that $\hat\Mc\leq \Maff$. Let $(\mu,\mu_\partial)$ be viable for $\Maff$. 
 Define $y(x)$ and $z(x)$ for $x\in \Omega$ as in \eqref{eq:defyz}; then the same arguments as in the proof of Theorem \ref{thm:main} show that $y(\cdot)$ is a $W^{1,p}$ function on $\Omega$ whose weak derivative is $Dy(x)=z(x)$ for almost every $x\in\Omega$, as well as the fact that the average of $y$ with respect to $\mu_\partial$ coincides with the trace $y(\cdot)|_{\partial\Omega}$. Moreover, $(y(x),Dy(x))\in \conv K(x)$ for almost every $x\in\Omega$, since it is an average of points in $(\operatorname{supp}\mu)\cap(\{x\}\times Y\times Z)\subseteq K(x)$, and similarly $y(x)\in\conv K_\partial(x)$ for a.e.\ $x\in \partial \Omega$. By the Jensen inequality, we have
 \begin{gather*}
 \int_{\Omega\times Y\times Z}L(x,y,z)\,d\mu(x,y,z)\,d\mu\geq \int_{\Omega\times Y\times Z}\hat L(x,y,z)\,d\mu(x,y,z)\geq \int_{\Omega}\hat L(x,y(x),Dy(x))\,dx.
 \end{gather*}
 This shows that $\hat\Mc\leq \Maff$. 
 \end{proof}

\begin{example}[With integral constraints we may have $\hat \Mc<\Maff$; nonconvex $Y$]\label{ex:integralconstraints} Let $Y\subset \R^2$ be the set containing only the four points $(y_1,y_2)$ with $y_i\in\{-1,1\}$. Let $\Omega=(0,1)\subset\R$ and $Z=\overline{B(0,1)}\subset\R^2$ the closed unit disc. Thus $K(x)=\{(\pm1,\pm1)\}\times \overline{B(0,1)}$ for all $x\in (0,1)$. Let, for $(x,y,z)\in \Omega\times Y\times Z$ and writing $y=(y_1,y_2)$,
\[L(x,y,z)=y_1y_2+\|z\|^2,\qquad C(x,y,z)=-y_1y_2.\]
Observe that $y_1y_2$ only takes the values $\pm1$ on $Y$.
The convexifications in $(y,z)$ are equal to $\conv Y=[-1,1]^2$, $\conv Z=Z=\overline{B(0,1)}$, $\conv K(x)= [-1,1]^2\times \overline{B(0,1)}$, and, for $x\in (0,1)$ and $(y,z)\in\conv K(x)$,
\[\hat L(x,y,z)=|y_1+y_2|-1+\|z\|^2,\qquad \hat C(x,y,z)=|y_2-y_1|-1.\]
Consider the problem 
\begin{alignat}{2}
 \hat \Mc =&
 \inf_{y \in W^{1,\infty}(\Omega,\conv Y)} 
  &\quad&
 \int_{\Omega}\hat L(x,y(x),Dy(x))dx
\label{opt:convexified2c}\\
  &\text{subject to}&& 
   (y(x),Dy(x))\in \conv K(x),\quad \text{a.e.\ }x\in \Omega, \nonumber\\
   &&&\int_{\Omega}\hat\Cc(x,y(x),Dy(x))\,dx\leq 0. \nonumber
\end{alignat}
This problem is analogous to \eqref{opt:convexified2b} in the bulk, with the addition of the integral constraint. Let us compare it with the problem \eqref{opt:relaxed} of finding $\Maff$ (with $L_\partial=A_i=A_{\partial,i}=B_i=B_{\partial,i}=C_{\partial,i}=0$). If $y\colon(0,1)\to\R^2$ is a constant curve whose image is a point in the segment joining $(-1,1)$ to $(1,-1)$, then the integral of $\hat L$ takes its minimal value of $-1$, and since the points in that segment that are close to the origin satisfy $\hat C(x,y,z)<0$, we conclude that $\hat\Mc=-1$. On the other hand, a measure $\mu$ viable in \eqref{opt:relaxed} we have
\begin{multline*}
    \int_{\Omega\times Y\times Z} L(x,y,z)\,d\mu(x,y,z)=\int_{\Omega\times Y\times Z} y_1y_2+\|z\|^2\,d\mu(x,y,z)\\
    =\int_{\Omega\times Y\times Z} -C(x,y,z)+\|z\|^2\,d\mu(x,y,z).
\end{multline*}
Thus in order to satisfy the constraint that the integral of $C$ be $\leq 0$, the integral of $L$ must be $\geq 0$. The  integral of $L$ achieves its minimal value of 0 by letting $\mu$ be uniformly distributed on $\Omega\times Y\times \{0\}$. This means that $\Maff= 0>-1=\hat\Mc$.
\end{example}

\begin{example}[With integral constraints we may have $\hat \Mc<\Maff$; nonconvex $Z$]\label{ex:integralconstraints2} 
We now modify Example \ref{ex:integralconstraints} slightly to show that a similar situation can occur when the nonconvexity appears in the set of allowable velocities $Z$.

Let $Z\subset \R^2$ be the set containing only the four points $(z_1,z_2)$ with $y_i\in\{-1,1\}$. Let $\Omega=(0,1)\subset\R$ and $Y=\overline{B(0,1)}\subset\R^2$ the closed unit disc. Thus $K(x)=\overline{B(0,1)}\times \{(\pm1,\pm1)\}$ for all $x\in (0,1)$. Let, for $(x,y,z)\in \Omega\times Y\times Z$ and writing $z=(z_1,z_2)$,
\[L(x,y,z)=z_1z_2,\qquad C(x,y,z)=-z_1z_2.\]
Observe that $z_1z_2$ only takes the values $\pm1$ on $Z$.
The convexifications in $(y,z)$ are equal to $\conv Y=Y=\overline{B(0,1)}$, $\conv Z=[-1,1]^2$, $\conv K(x)= \overline{B(0,1)}\times [-1,1]^2$, and, for $x\in (0,1)$ and $(y,z)\in\conv K(x)$,
\[\hat L(x,y,z)=|z_1+z_2|-1,\qquad \hat C(x,y,z)=|z_2-z_1|-1.\]
Consider the problem \eqref{opt:convexified2c}. Let us compare it with the problem \eqref{opt:relaxed} of finding $\Maff$ (with $L_\partial=A_i=A_{\partial,i}=B_i=B_{\partial,i}=C_{\partial,i}=0$). If $y\colon(0,1)\to\R^2$ is a  curve whose derivatives $Dy(x)$ are, for almost every $x\in(0,1)$, contained in the segment joining $(-1,1)$ to $(1,-1)$, then the integral of $\hat L$ takes its minimal value of $-1$, and since the points in that segment that are close to the origin satisfy $\hat C(x,y,z)<0$, we conclude that $\hat\Mc=-1$; for an explicit example, one may take the curve 
\[
 y(x)=
  \begin{cases}
  (x,-x),&x\in (0,\tfrac12),\\
  (\tfrac12-x,-\tfrac12+x),&x\in[\tfrac12,1).
  \end{cases}
\]
On the other hand, a measure $\mu$ viable in \eqref{opt:relaxed} we have
\begin{equation*}
    \int_{\Omega\times Y\times Z} L(x,y,z)\,d\mu(x,y,z)=\int_{\Omega\times Y\times Z} z_1z_2\,d\mu(x,y,z)
    =-\int_{\Omega\times Y\times Z} C(x,y,z)\,d\mu(x,y,z).
\end{equation*}
Thus in order to satisfy the constraint that the integral of $C$ be $\leq 0$, the integral of $L$ must be $\geq 0$. The  integral of $L$ achieves its minimal value of 0 by letting $\mu$ be uniformly distributed on $\Omega\times  \{0\}\times Z$. This means that $\Maff= 0>-1=\hat\Mc$.
\end{example}

\section{Optimal control}
\label{sec:optimalcontrol}

{This section extends our results to optimal control. We first formulate the results under as general assumptions as possible and then present their simplified versions. We shall denote the control inside $\Omega$ by $u:\Omega\to U$ and on the boundary by $u:\partial\Omega\to U_\partial$. We shall assume that $U$ and $U_\partial$ are complete separable metric spaces (e.g., closed subsets of a Euclidean space). The function $\pi_{\Omega\times Y\times Z} : \Omega\times Y\times Z\times U \to \Omega\times Y\times Z$ will denote the projection on $\Omega\times Y\times Z$,. i.e., $\pi_{\Omega\times Y\times Z}(x,y,z,u) = (x,y,z)$. Similarly $\pi_{\partial\Omega\times Y}:\partial\Omega\times Y\times U_\partial \to \partial\Omega\times Y$ is defined by $\pi_{\partial\Omega\times Y}(x,y,u) = (x,y)$}. 

Let  $\A_i,\B_i\colon\Omega\times Y\times Z\times U\to\R$ and $\Cc_i\colon\Omega\times Y\times Z\to\R$ be Borel measurable functions for all $i$ in a (possibly uncountably-infinite) index set $I$. Similarly, let $\A_{\partial,i},\B_{\partial,i}\colon\partial\Omega\times Y\times U_\partial\to\R$ and $\Cc_{\partial,i}\colon\partial\Omega\times Y\to\R$ be Borel measurable functions for all $i\in I$. Let also $L\colon\Omega\times Y\times Z\times U\to\R$ and $L_\partial\colon\partial\Omega\times Y\times U_\partial\to\R$ be measurable and locally bounded functions.



\subsection{Problem statement}
For $p\in[1,+\infty]$, consider the optimal control problem
\begin{alignat*}{2}
 \Mc^{\mathrm{oc}}=
 \inf_{\substack{y\colon\Omega\to Y\\u\colon\Omega\to U\\ u_\partial\colon\partial\Omega\to U_\partial}} 
  &\quad&&
 \int_{\Omega}L(x,y(x),Dy(x),u(x))dx
  +
 \int_{\partial\Omega} L_\partial(x,y(x),u_\partial(x))\,d\sigma(x)\\
  \text{subject to}&&& y\in W^{1,p}(\Omega;Y),
 \quad u,u_\partial\text{ are measurable},\\
  &&& \A_i(x,y(x),Dy(x),u(x))=0,\quad\B_i(x,y(x),Dy(x),u(x))\leq 0,\quad x\in\Omega,\;i\in I,\\
  &&& \A_{\partial,i}(x,y(x),u(x))=0,\quad\B_{\partial,i}(x,y(x),u_\partial(x))\leq 0,\quad x\in\partial\Omega,\;i\in I,\\
  &&& \int_{\Omega}\Cc_i(x,y(x),Dy(x))\,dx\leq 0,\quad \int_{\partial\Omega}\Cc_{\partial,i}(x,y(x))\,dx\leq 0,\quad i\in I.
\end{alignat*}
Consider also its relaxation, for $p\in [1,+\infty)$:
 \begin{alignat*}{2}
   M_{\mathrm r}^{\mathrm{oc}}=
   \inf_{\substack{\mu \in\mathcal M(\Omega\times Y\times Z\times U)\\
   \mu_\partial \in\mathcal M( \partial\Omega\times Y\times U_\partial)
   }} & \quad&&\int_{\Omega\times Y\times Z\times U}\hspace{-5mm}L(x,y,z,u)d\mu(x,y,z,u)
   +\int_{\partial\Omega\times Y\times U_\partial}\hspace{-5mm}L_\partial(x,y,u)\,d\mu_\partial(x,y,u)\\
   \text{subject to}&&& \textstyle\operatorname{supp}\mu\subseteq\bigcap_{i\in I}\A_i^{-1}(0)\cap \B_i^{-1}((-\infty,0]),\\
    &&&\textstyle\operatorname{supp}\mu_\partial\subseteq\bigcap_{i\in I}\A_{\partial,i}^{-1}(0)\cap\B_{\partial,i}^{-1}((-\infty,0]),\\
&&& \textstyle\int_{\Omega\times Y\times U}\Cc_i(x,y,z)\,d\mu(x,y,z,u)\leq 0,\\
   &&&\textstyle\int_{\partial\Omega\times Y\times U_\partial}\Cc_{\partial,i}(x,y)\,d\mu_\partial(x,y,u)\leq 0,\quad i\in I,\\   
   &&&\textstyle\mu(\Omega\times Y\times Z\times U)=|\Omega|,\\
   &&&\textstyle\int_{\Omega\times Y\times Z\times U}\|y\|^p+\|z\|^p\,d\mu(x,y,z,u)<+\infty,\\
   &&&\textstyle\int_{\partial\Omega\times Y\times U_\partial}\|y\|^p\,d\mu_\partial(x,y,u)<+\infty,\\
      &&&\textstyle\int_{\Omega\times Y\times Z\times U}\frac{\partial\phi_1}{\partial x}+y^\top\frac{\partial\phi_2}{\partial x}+\phi_2^\top z\,d\mu=\int_{\partial\Omega\times Y\times U_{\partial}}(\phi_1+y^\top\phi_2)\mathbf 
      n\,d\mu_\partial\\
    &&&\hspace{4cm}\forall\phi_1\in C^\infty(\Omega; \R),\;\phi_2\in C^\infty(\Omega;\R^m).
 \end{alignat*}
 For $p=+\infty$, the constraints involving integrals of $\|\cdot\|^p$ are replaced by the assumption that the supports of $\mu$ and $\mu_\partial$ be compact. 
 
 \begin{remark}\label{rmk:affineOC}
 The measures involved in $M_{\mathrm r}^{\mathrm{oc}}$ satisfy a Liouville equation for affine test functions only, that is, for functions $\phi(x,y)=\phi_1(x)+\phi(x)y$, and are thus analogous to the affinely relaxed occupation measures $\OM_p$; cf. Remark \ref{rmk:liouville}. A corollary to Theorem \ref{thm:optimalcontrol} analogous to Corollary \ref{coro:unrelax} holds, although for brevity we do not state it explicitly.
 \end{remark}

\subsection{General formulation}
The functions $L$ and $L_\partial$ have close relatives that do not depend on $U$ and $U_\partial$ and that will play a role in our assumptions: let the functions $\bar L\colon\Omega\times Y\times Z\to\R$ and $\bar L_\partial\colon\partial\Omega\times Y\to\R$ be defined by
 \begin{gather*}
  \bar L(x,y,z)=\inf\{L(x,y,z,u):u\in U,\;\A_i(x,y,z,u)=0,\;\B_i(x,y,z,u)\leq 0,\;i\in I\}\\
  \bar L_\partial(x,y)=\inf\{L_\partial(x,y,u):u\in U_\partial,\;\A_{\partial,i}(x,y,u)=0,\;\B_{\partial,i}(x,y,u)\leq 0, \;i\in I\}.
 \end{gather*}
 
Assume the following:
\begin{enumerate}[label=OC\arabic*.,ref=OC\arabic*]
 \item \label{oc:first} 
 \label{oc:convexcond1} (Fiberwise convexity of the constraints) We assume:
 \begin{enumerate}
 \item$\pi_{\Omega\times Y\times Z}(\bigcap_{i\in I}\A_i^{-1}(0)\cap\B_i^{-1}((-\infty,0]))\cap\{x\}\times Y\times Z$ is nonempty and convex for all $x\in \Omega$.
 \item
 $\pi_{\partial\Omega\times Y}(\bigcap_{i\in I}\A_{\partial,i}^{-1}(0)\cap \B_{\partial,i}^{-1}((-\infty,0]))\cap\{x\}\times Y$ is nonempty and convex for all $x\in\partial\Omega$ and all $i\in I$.
 \item
$\Cc_i(x,y,z)$ is convex in $(y,z)$, and
\item $\Cc_{\partial,i}(x,y)$ is convex in $y$, for all $x\in \Omega$ and all $i\in I$.
 \end{enumerate}
 
 \item \label{oc:closedsublevels} 
   (Fiberwise minima are attained by the controls)
   The functions $L$ and $L_{\partial}$ reach their minima in $U$ and $U_\partial$, respectively, once we fix $(x,y,z)$; in other words, the sets
   \begin{gather*}
    \operatornamewithlimits{arg\,min}_{\substack{u\in U\\\A_i(x,y,z,u)=0\\\B_i(x,y,z,u)\leq 0}}L(x,y,z,u)\qquad\text{and}\qquad\operatornamewithlimits{arg\,min}_{\substack{u\in U_\partial \\\A_{\partial,i}(x,y,u)=0\\\B_{\partial,i}(x,y,u)\leq 0}}L_{\partial}(x,y,u),
   \end{gather*}
    are nonempty and closed for all  $(x,y,z)\in\Omega\times Y\times Z$ and $(x,y)\in \partial\Omega\times Y$, respectively.
 \item\label{oc:Lconvex} (Convexity of the Lagrangian densities)
 \begin{enumerate}
     \item 
 $\bar L(x,y,z)$ is locally bounded and convex in $(y,z)$ for each fixed $x\in\Omega$.
 \item
 $\bar L_\partial(x,y)$ is locally bounded and convex in $y$ for each fixed $x\in\partial\Omega$.
\end{enumerate}
 \label{oc:last}
\end{enumerate}

\subsection{Simplified formulation 1}

Here is a simple situation that satisfies assumptions \ref{oc:first}--\ref{oc:last} above.  Let $p=+\infty$, $\Omega\subset\R^n$ be as above, $Y\subseteq \R^m$ and $Z\subseteq \R^{m\times n}$ be compact, convex sets with nonempty interior. Let $U$ and $U_\partial$ be compact and convex subsets of $\R^d$ for some $d\in\N$. Assume that,  for each $x\in\Omega$, $L(x,y,z,u)$, $B_i(x,y,z,u)$, and $C_i(x,y,z,u)$ are convex in $(y,z,u)$, and that $A_i(x,y,z,u)$ is affine in $(y,z,u)$. Then we are in the setting concerning our results below. Similar assumptions on the boundary Lagrangian density $L_\partial$ and in the boundary constraints $A_{\partial,i}$, $B_{\partial,i}$, and $C_{\partial,i}$ are also possible within our framework.

\subsection{Simplified formulation 2}

Here is another situation that satisfies assumptions \ref{oc:first}--\ref{oc:last} above. For simplicity we will ignore the boundary part of the problem. Let $p=+\infty$, $\Omega\subset\R^n$ be as above, $Y\subseteq \R^m$ and $Z\subseteq \R^{m\times n}$ be compact, convex sets with nonempty interior. Let $U$ be compact subset of $\R^n$, and let $f\colon \Omega\times Y\times U\to Z$ be a continuous function such that, for each $x\in\Omega$, the image of the map $(x,y,u)\mapsto (y,f(x,y,u))$ is convex\footnote{This is true for example when $f$ depends only on $(x,u)$ and $f(\{x\}\times U)$ is convex for each $x\in\Omega$.}. Let $\ell\colon\Omega\times Y\times U\to \R$ be a continuous Lagrangian density such that 
\[ 
 \bar \ell(x,y,z)=\inf_{\substack{u\in U\\f(x,u)=z}} \ell(x,y,u)
\]
is convex in $(y,z)$. 
Consider the optimal control problem
\begin{alignat*}{2}
\inf_{(y,u)}&\quad&&\int_{\Omega}\ell(x,y(x),u(x))\,dx\\
\text{subject to}&&&
Dy(x)=f(x,y(x),u(x)),\quad\text{a.e.\ $x\in\Omega$},\\
&&&a_i(x,y(x),f(x,y(x),u(x))=0,\quad i\in I,\text{a.e.\ $x\in\Omega$,}\\
&&&b_i(x,y(x),f(x,y(x),u(x)))\leq 0,\quad i\in I,\text{a.e.\ $x\in\Omega$,}\\
&&&\int_\Omega c_i(x,y(x),f(x,y(x),u(x)))\,dx\leq 0,\quad i\in I,
\end{alignat*}
where $a_i,b_i,c_i\colon\Omega\times Y\times Z\to\R$ are continuous and satisfy
$a_i|_{\{x\}\times Y\times Z}^{-1}(0)$ is convex for each $x\in \Omega$, and
$b_i$ and $c_i$ are convex in $(y,z)$ for fixed $x\in\Omega$.

To model this situation in our setup above, we add a constraint of the form 
\[A_0(x,y,z,u)=\operatorname{dist}(z,f(x,y,u))\]
so that the constraint $A_0= 0$ means that the velocity $z$ is controlled by $u$ via $z\in f(x,u)$. We let, for $i\in I$, $i\neq 0$,
\begin{gather*}
    A_i(x,y,z,u)=a_i(x,y,z),\qquad B_i(x,y,z,u)=b_i(x,y,z),\\
    C_i(x,y,z,u)=c_i(x,y,z), 
\end{gather*}
and
\[L(x,y,z,u)=\ell(x,y,u)\qquad\text{and}\qquad\bar L(x,y,z)=\bar\ell(x,y,z).\]
With these assumptions, one can check easily that \ref{oc:first}--\ref{oc:last} hold.

\subsection{Optimal control -- main result}

In the optimal control context outlined above, we have the following theorem.
 
 \begin{theorem}\label{thm:optimalcontrol}
  In the setting described above, with assumptions \ref{oc:first}--\ref{oc:last}, we have $\Mc^{\mathrm{oc}}=M_r^{\mathrm{oc}}$.
 \end{theorem}
 \begin{proof}
 To see that $\Mc^{\mathrm{oc}}\geq M_r^{\mathrm{oc}}$, note that for every triad of functions $(y,u,u_\partial)$ that are minimization candidates in $\Mc^{\mathrm{oc}}$, we can produce a pair of measures $(\mu,\mu_\partial)$ defined by
 \begin{align*}
   \int_{\Omega\times Y\times Z\times U}f(x,y,z,u)d\mu(x,y,z,u)&=\int_{\Omega}f(x,y(x),Dy(x),u(x))\,dx,\quad f\in C^0(\Omega\times Y\times Z\times U),\\
   \int_{\partial\Omega\times Y\times U_\partial}f(x,y,u)\,d\mu_\partial(x,y,u)&=\int_{\partial\Omega}f(x,y(x),u_\partial(x))\,d\sigma(x),\quad f\in C^0(\partial\Omega\times Y\times U_\partial);
 \end{align*}
 then it is immediate that the pair $(\mu,\mu_\partial)$ is a minimization candidate in $M_{\mathrm{r}}^{\mathrm{oc}}$ that satisfies all the relevant constraints, and also satisfies
 \[\int_{\Omega\times Y\times Z\times U} L\,d\mu+\int_{\partial\Omega\times Y\times U_\partial}L_\partial\,d\mu_\partial=\int_{\Omega}L(x,y(x),Dy(x),u(x))\,dx+\int_{\partial\Omega}L_\partial(x,y(x),u_\partial(x))\,d\sigma(x).\]

  Let us see that  $\Mc^{\mathrm{oc}}\leq M_r^{\mathrm{oc}}$. Pick a minimization candidate $(\mu,\mu_\partial)$ for $M_{\mathrm{r}}^{\mathrm{oc}}$ satisfying all the constraints of that problem. 
 As in the proof of Theorem \ref{thm:main}, disintegrate
 \[\mu(x,y,z,u)=\nu(y,z,u\mid x)\,dx,\quad \mu_\partial(x,y,u)=\nu_\partial(y,u\mid x)\,d\sigma(x),\]
 and define
 \begin{gather*}
  y(x)=\int_{Y\times  Z\times U}y\,d\nu(y,z,u\mid x),\quad z(x)=\int_{Y\times Z\times U}z\,d\nu(y,z,u\mid x),\\
  y_\partial(x)=\int_{Y\times U_\partial}y\,d\nu_\partial(y,u\mid x).
 \end{gather*}
 Then the same arguments as in the proof of Theorem \ref{thm:main} show that $z(\cdot)$ is in $L_p$, that $Dy(\cdot)=z(\cdot)$ weakly, that $y(\cdot)$ is in $W^{1,p}$, and that $y_\partial(\cdot)$ is the Sobolev trace of $y(\cdot)$. 
 
 We pick measurable functions $u(\cdot)\colon\Omega\to U$ and $u_\partial(\cdot)\colon\partial\Omega\to U_\partial$ satisfying 
 \begin{align}
 \label{eq:argmin1}
 &u(x)\in \operatornamewithlimits{arg\,min}_{\substack{u\in U\\\A_i(x,y(x),z(x),u)=0\\\B_i(x,y(x),z(x),u)\leq 0}}L(x,y(x),z(x),u),\qquad x\in \Omega,\\
 \label{eq:argmin2}
 &u_\partial(x)\in \operatornamewithlimits{arg\,min}_{\substack{u\in U_\partial \\\A_{\partial,i}(x,y_\partial(x),u)=0\\\B_{\partial,i}(x,y_\partial(x),u)\leq 0}}L_{\partial}(x,y_\partial(x),u),\qquad x\in \partial \Omega.
 \end{align}
 Let us explain why we can do this. The set-valued maps induced by the right-hand sides of \eqref{eq:argmin1} and \eqref{eq:argmin2} $\Omega\rightrightarrows U$ and $\partial\Omega\rightrightarrows U_\partial$ have nonempty and closed images by \ref{oc:closedsublevels} and have measurable graphs, so we may apply a Characterization Theorem \cite[Th.\ 8.1.4]{aubinfrankowska} to conclude that they are weakly measurable.
 Then  we may apply the Kuratowski--Ryll-Nardzewski Selection Theorem (see \cite[Th. 18.13]{aliprantisborder} or \cite[Th.\ 8.1.3]{aubinfrankowska}) to get the functions $u(x)$ and $u_\partial(x)$.

The arguments used in the proof of Theorem \ref{thm:main} to prove that the constraints involving $\A_i,\B_i,\Cc_i,\A_{\partial,i},\B_{\partial_i},\Cc_{\partial_i}$ are satisfied by the points $(x,y(x),z(x))$, $x\in\Omega$, and $(x,y(x))$, $x\in\partial\Omega$,  carry through to the present situation with minimal modifications for the points $(x,y(x),z(x),u(x))$, $x\in\Omega$, and $(x,y(x),u_\partial(x))$, $x\in\partial\Omega$, using \ref{oc:convexcond1}.
 
 Thus with this choice, the triad $(y(\cdot), u(\cdot), u_\partial(\cdot))$ is a minimization candidate for $\Mc^{\mathrm{oc}}$, and by Jensen's inequality and \ref{oc:Lconvex} we have
 \begin{align*}
  \Mc^{\mathrm{oc}}
  &\leq \int_{\Omega} L(x,y(x),z(x),u(x))\,dx+  \int_{\Omega_\partial} L_\partial(x,y(x),u_\partial(x))\,d\sigma(x)\\
  &= \int_{\Omega} \bar L(x,y(x),z(x))\,dx+  \int_{\Omega_\partial} \bar L_\partial(x,y(x))\,d\sigma(x)\\
  &\leq\int_{\Omega\times Y\times Z\times U} \bar L(x,y,z)\,d\mu(x,y,z,u)
  +  \int_{\Omega_\partial\times Y\times U_\partial} \bar L_\partial(x,y)\,d\mu_\partial(x,y,u)\\
  &\leq\int_{\Omega\times Y\times Z\times U} L(x,y,z,u)\,d\mu(x,y,z,u)
  +  \int_{\Omega_\partial\times Y\times U_\partial} L_\partial(x,y,u)\,d\mu_\partial(x,y,u),
 \end{align*}
 which implies indeed that $\Mc^{\mathrm{oc}}\leq M_{\mathrm{r}}^{\mathrm{oc}}$.
 \end{proof}

\section{Application to micromagnetics}\label{sec:magnetics}
In this section we describe how variational  problems appearing  in  micromagnetics fit our convexity assumptions and hence can be solved with our linear formulation with occupation measures.

We confine ourselves to the two-dimensional situation.  However, results of this section hold also  in the $N$-dimensional case, too, for any $N\ge 2$. See e.g. \cite{pedregal0}.
In the classical theory of rigid  ferromagnetic bodies, based mainly on \cite{Brown}, a {\it magnetization } $m\colon\O\to\R^2$, describing the state of the body $\O\subset\R^2$,
depends on a position $x\in\O$ and has a given temperature-dependent magnitude
$$
|m(x)| =w(T)\ \ \mbox{for almost all $x\in\O$},$$
where $w(T)=0$ for $T\ge T_c$ the so-called Curie point. Let us treat the case when the temperature is fixed below the Curie point
and thus we shall assume that $|m|=1$ almost everywhere in $\O$.
In the so-called no-exchange formulation, the energy of a large rigid ferromagnetic body $\O\subset \R^2$ consists of three parts and the variational principle governing equilibrium configurations
can be stated (see e.g. \cite{jam-kin} and references therein) as the problem of minimizing
\begin{eqnarray}\label{microfunct}
E(m)=\int_\O \ff(m(x))\,dx-\int_\O H(x)\cdot m(x)\,dx+\frac{1}{2}\int_{\O}\nabla u_m(x)\cdot m(x)\,dx \end{eqnarray}
with respect to $m$ and $u_m$, where $f\colon \mathcal{S}^1\to [0;+\infty)$ is continuous function  defined on the unit sphere  in $\R^2$ centered at the origin $\mathcal{S}^1$. Further,  $m:\O\to\R^2$ satisfies
\begin{equation}\label{circleconstraint}|m(x)|=1\qquad \text{a.e. $x\in\O$,}
\end{equation}
$H\colon\O\to\R^2$ is an applied external magnetic field  and $u_m:\O\to\R$ is a potential  of an induced magnetic field. The first term in \eqref{microfunct} is an anisotropy energy with a density $\ff$ which is supposed to be an
 even nonnegative function depending on material properties and
 exhibiting crystallographic symmetry. The second term is an
 interaction energy and the last term is a magnetostatic energy coupled
 with the magnetization field through the equation
\begin{eqnarray}\label{maxwell}
\mbox{div}(-\nabla u_m + m\chi_\O)=0 \ \mbox{ in } \R^2,\end{eqnarray}
where $\chi_\O:\R^2\to\{0,1\}$ is the characteristic function of $\O$.
Equation \eqref{maxwell} is interpreted in the classical weak sense, i.e., 
\begin{equation}\label{eq:weakdiv}
\int_{\R^2} \nabla \phi\cdot(-\nabla u_m + m\chi_\O)\,dx=0\qquad \forall\,\phi \in C_0^\infty(\R^2).
\end{equation}




Let us 
denote 
$$
\add\coloneqq\{ m\in L^2(\O;\R^2);\ |m(x)|=1 \mbox{ for almost all $x\in\O$}\}.$$
Eventually, we are concerned with the problem of minimizing
\begin{eqnarray}\label{equivformulation}
E(m)=\int_\O \ff(m(x))\, dx-\int_\O H(x)\cdot m(x)\,dx
+\frac{1}{2}\int_{\O}\nabla u_m(x)\cdot m(x)\,dx
\end{eqnarray}
with respect to $m$ and $u_m$, and  subject to  \eqref{eq:weakdiv}. It was shown in \cite{jam-kin} that if $N=1$   the infimum of \eqref{equivformulation} is not necessarily attained even if $H=0$. The reason is that minimizing sequences of magnetizations    oscillate between $\tilde m_1$ and $-\tilde m_1$  keeping the divergence of the magnetization small, so that the gradient of $u_m$ is negligible.  This construction leads to a minimizing sequence  that weakly converges to zero in $L^2(\O;\R^2)$  and at the same time it shows that $\inf_{m \in \add} E(m)=0$. However, no minimizer exists and  $m=0$ does not belong to $\add$.   Therefore it is desirable to construct the largest lower semicontinuous envelope of $E$ from  \eqref{equivformulation} and there are well known techniques how to do that \cite{d04,pedregalenv}. One uses parameterized Young measures, the other one
extends $\ff$ by infinity outside $\mathcal{S}^1$ and then calculates the convex envelope of this newly defined function. Obviously, the convex envelope is  finite only on the closed unit ball in $\R^2$.
The relaxed problem is convex but not linear,  since the stray-field energy, i.e. the last term in \eqref{equivformulation} is quadratic.

\subsection{Formulation with Young measures}

As  $\add$ is not convex  we cannot rely
on direct methods \cite{d04} to prove the existence of a solution.
 In fact, the solution to (\ref{microfunct}) need not exist in $\add$, cf. \cite{jam-kin} for the case $H=0$. Therefore, it makes sense to look for a relaxation of the problem. It is known \cite{pedregal0} that this can be done using  Young measures \cite{y69}.  Moreover, Young measures competing in the minimization problem are supported on $\cs$, i.e. on a compact set.
The relaxed problem consists of minimizing
 \begin{align}\label{relaxmic}
E_Y(\nu):=\int_\O\int_{\mathcal{ S}^1}\ff(A)\nu_x(dA)\,dx -\int_\O H(x)\cdot m(x)\,dx +\frac{1}{2}\int_{\O}\nabla u_m(x)\cdot m(x)\,dx ,\end{align}
with respect to $\nu$ and $u_m$,
subject to \eqref{eq:weakdiv} with  $\ m(x)=\int_{\mathcal{ S}^1} A\nu_x(dA)$ for a.a. $x\in\O $ and $\nu \in \bar\add := \mathcal{Y}^\infty(\O;\mathcal{S}^1)$, where $\mathcal{Y}^\infty(\O;\mathcal{S}^1)$ denotes the space of Young measures such that $\nu_x$ is a probability measure on $\mathcal{S}^1$ for a.a. $x \in \O$ and the dependence of $\nu_x$ on $x$ is measurable.

 On the other hand,  it was shown in \cite[Thm.~3.4, Remark 3.7]{desim}  that if $\ff: {\mathcal{ S}^1}\to [0,+\infty)$ and $H\in L_{\rm loc}^2(\O;\R^2)$ then 
 \begin{equation}\label{eq:desimone} 
 \inf_{m\in\add}E(m) =\min_{\nu\in\bar\add} E_Y(\nu)=\min_{m\in\add^{**}}E^{**}(m)
 \end{equation}
 where
$$
\add^{**}=\{ m\in L^2(\O;\R^2);\ | m(x)|\le 1\ \mbox{ for
almost all $x \in \O$}\} $$ and
\begin{eqnarray}\label{relfci}
E^{**}(m)=\int_\O
\hat\ff^{**}(m(x))\, dx-
\int_\O H(x)\cdot m(x)\, dx+\frac{1}{2}\int_{\O}\nabla u_m(x)\cdot m(x)\, dx,\end{eqnarray}
and $\hat\ff^{**}:\R^2\to\R$, $\hat\ff^{**}=\sup\left\{f;\ f\le\hat\ff\ ,\
 \mbox{ $f$ is convex}\right\}$
 is the convex envelope of $\hat\ff:\R^{2}\to\R\cup\{+\infty\}$,
\begin{equation*}
 \hat\ff(m) =\left\{\begin{array}{ll}
                         \ff(m) & \mbox{ if $|m|=1$}\\
                           +\infty & \mbox{ otherwise.}
                          \end{array}
                   \right. 
\end{equation*}
Again, $u_m$ in \eqref{relfci} is calculated from 
\eqref{eq:weakdiv}
The functional $E^{**}$ is sequentially weakly lower semicontinuous on $\add^{**}$
 and it possesses a minimum on $\add^{**}$. 
 We again emphasize that both \eqref{relaxmic} and \eqref{relfci} are convex but nonlinear. The formulation of these  problems in terms of occupation measures leads to a linear problem  that makes it attractive for numerical solution. Similarly, this applies to some problems in linear and nonlinear elasticity where the lower semicontinuous envelope of the original problem is convex; see \cite{kohn,raoult} for some examples.




\subsection{Formulation with occupation measures}
\label{sec:occupationmeasureformulation}

To formulate the problem \eqref{equivformulation} in terms of occupation measures following Definition \ref{def:M2}),  let $\Omega$  be a subset of $\R^2$ containing $\O$, and let $Y=\R^3$ and $Z=\R^{2\times 3}$. Right away, we see that the problem \eqref{equivformulation} with constraints \eqref{circleconstraint} and \eqref{eq:weakdiv} can be modelled as
\begin{alignat}{3}
\label{prob1} & &\min_{ (\mu,\mu_\partial) \in\OMold_2}& \quad&&\int_{\Omega\times Y\times Z}f(y_1,y_2)-H(x)\cdot(y_1,y_2)+\frac12(z_3\cdot(y_1,y_2))\,d\mu(x,y,z)\\
\label{prob2} &&\text{subject to}&&& \operatorname{supp} \mu\subseteq \Omega\times (S_1\times \R)\times Z\subset \Omega\times Y\times Z,\\
\label{prob3} &&&&& \int_{\Omega\times Y\times Z}\nabla \phi(x)\cdot(-z_3+(y_1,y_2)\chi_\O(x))\,d\mu(x,y,z)=0,\quad\phi\in C^1_0(\Omega).
\end{alignat}
Observe that the solutions of \eqref{equivformulation} are of the form $y(x)=(m(x),u(x))$, where $m(x)$ is in the unit circle $S^1\subset\R^2$ for all $x\in\Omega$, and $u(x)\in\R$. In the formulation above, $(y_1,y_2)$ correspond to $m(x)$ and $y_3$ corresponds to $u(x)$, so that also $z_3$ corresponds to $\nabla u(x)$.

Let us now show how to put this in the form of our problem \eqref{opt:relaxedold}.
To get the integral in \eqref{prob1}, let
\[L(x,y,z)=f(y_1,y_2)-H(x)\cdot(y_1,y_2)+\frac12(z_3\cdot(y_1,y_2)),\quad (x,y,z)\in\Omega\times Y\times Z.\]
To model the constraint \eqref{prob2}, let
\[A(x,(y_1,y_2,y_3),z)=y_1^2+y_2^2-1.\]
To model the constraint \eqref{prob3}, we use the trick described in Example \ref{ex:constraints}: we use $I=C_0^1(\Omega)$ as index set and we let 
\[\Cc_\phi(x,y,z)=\nabla\phi(x)\cdot (-z_3+(y_1,y_2)\chi_\O(x)),\qquad \phi\in C_0^1(\Omega).\]
All other functions $L_\partial, A_{i,\partial},B_i,B_{i,\partial},C_{i,\partial}$ in \eqref{opt:relaxedold} can be set to 0.

Observe that even if $f$ is convex, Theorem \ref{thm:main} cannot be applied to \eqref{prob1}--\eqref{prob3} above because the constraint \eqref{prob2} is not convex. Instead, we can apply Theorem \ref{thm:convexificationb}, which tells us that if we take the affine occupation relaxation \eqref{opt:relaxed}, that is, if we minimize \eqref{prob1} over $\OM_2$ (rather than $\OMold_2$; see \eqref{eq:magnetizationaffine} below), then we get the same result $\hat M$ as in the convexification \eqref{opt:convexified2b}. The latter coincides with the problem of minimizing $E^{**}$ as in \eqref{relfci} over $\add^{**}$ and, by \eqref{eq:desimone}, with \eqref{equivformulation}. We have proved:

\begin{proposition}\label{prop:magnet_nogap}
 The problem \eqref{equivformulation} is equivalent to the (computationally tractable) problem
 \begin{alignat}{3}\label{eq:magnetizationaffine}
&&\min_{ (\mu,\mu_\partial) \in\OM_2}& \quad&&\int_{\Omega\times Y\times Z}f(y_1,y_2)-H(x)\cdot(y_1,y_2)+\frac12(z_3\cdot(y_1,y_2))\,d\mu(x,y,z)\\
 &&\mathrm{subject\ to}&&& \text{\eqref{prob2} $\mathrm{and}$ \eqref{prob3},} \nonumber
\end{alignat}
that is, the minimum in~\eqref{eq:magnetizationaffine} is equal to the infimum in \eqref{equivformulation}.
\end{proposition}

\begin{remark} Since the occupation measure relaxation \eqref{prob1} is sandwiched between the affinely relaxed problem \eqref{eq:magnetizationaffine} and the original problem \eqref{equivformulation}, it follows from  Proposition~\ref{prop:magnet_nogap} that these are all equal to each other. 
\end{remark}

\bigskip

\appendix
\section{Constancy lemma}

The following is essentially a version of the what is known as the Constancy Theorem in Geometric Measure Theory, and it is a slight generalization of  \cite[Lemma 2.9]{kr22}. We use it in the proof of Theorem \ref{thm:main}. 

\begin{lemma}\label{lem:lebesgueproj}
{ Let $\Omega$ be an  bounded, connected, open subset of $\R^n$.
 Let $\mu$ be a positive, compactly-supported, Radon measure on $\Omega$ such that $\mu(\Omega)=|\Omega|$ and
 \begin{equation}\label{eq:circulation} \int_{\Omega}\nabla \phi(x)\, d\mu(x)=0\end{equation}
 for all compactly-supported $\phi\in C^\infty_0(\Omega;\R)$.
 Then $\mu$ equals the $n$-dimensional Lebesgue measure on  $\Omega$.}
\end{lemma}
\begin{proof}

 Let $R\subset \Omega$ be a small parallelepiped, and let $\tau$ be a translation such that $\tau(R)\subset \Omega$. We will show that $\mu(R)=\mu(\tau(R))$, and since this will be true for all $R$ and all $\tau$, $\mu$ must be a positive multiple of Lebesgue on $\Omega$ \cite[Thm. 2.20]{rudin}. 
 Write $\tau$ as a finite composition of translations $\tau_i$ in the directions of the axes $x_1,\dots,x_n$,
 \[\tau=\tau_k\circ\tau_{k-1}\circ\dots\circ\tau_1.\]
 Denote $\tilde \tau_i=\tau_i\circ\tau_{i-1}\circ\dots\circ\tau_1$ and set $\tilde\tau_0$ equal to the identity. We assume $\tau_1,\dots,\tau_k$ have been chosen also in a such a way that the convex hull of $\tilde\tau_{i-1}(R)\cup\tilde\tau_{i}(R)$ is contained in $\Omega$ for each $i$. 
 For each $i=1,\dots,k$, let $j_i$ be such that $\tau_i$ is a translation in direction $x_{j_i}$. Recall that $\chi_{\tilde\tau_{i}(R)}$ is the indicator function of the translated rectangle $\tilde\tau_{i}(R)$, and let
 \begin{multline*}
  \psi_i(x_1,\dots,x_n)=\int_{-\infty}^{x_{j_i}}\chi_{\tilde\tau_{i}(R)}(x_1,\dots,x_{j_i-1},s,x_{j_i+1},\dots,x_n)\\
  -\chi_{\tilde\tau_{i-1}(R)}(x_1,\dots,x_{j_i-1},s,x_{j_i+1},\dots,x_n)ds.
 \end{multline*}
 Observe that
  \[\operatorname{supp}\phi_i=\overline{\operatorname{conv}}(\tilde\tau_{i-1}(R)\cup\tilde\tau_i(R)),\]
 which is a compact set properly contained in $\Omega$.
 Approximating with smooth, compactly-supported functions and using the Lebesgue dominated convergence theorem, we conclude that  \eqref{eq:circulation} is true for $\phi=\psi_i$, which means, for the $j_i^{\mathrm{th}}$ entry,
 \begin{equation*}
 \mu(\tilde\tau_{i}(R))- \mu(\tilde\tau_{i-1}(R))=\int_{\Omega}\chi_{\tilde\tau_i(R)}-\chi_{\tilde\tau_{i-1}(R)}\,d\mu 
  =\int_{\Omega}\frac{\partial \psi_i}{\partial x_{j_i}}\,d\mu
  =0.
 \end{equation*}
 By induction we get
 \[\mu(R)=\mu(\tilde\tau_{0}(R))=\mu(\tilde\tau_{k}(R))=\mu(\tau(R)).\]
 Thus $\mu=c\,dx$ for some $c>0$. Since $\mu(\Omega)=|\Omega|$, we conclude that $c=1$. 
 \end{proof}

\section{Technical lemma for the definition of occupation measures}

\begin{lemma}\label{lem:integrability}
Let $p\in[1,+\infty)$, and let $\mu$ and $\mu_\partial$ be measures satisfying the second item in Definition \ref{def:M}.
Then for all $\phi\in \testfunctions_p$ we have that the functions 
 \[\phi(x,y)\mathbf n(x)\quad\text{and}\quad \frac{\partial\phi}{\partial x}(x,y)+z^\top\frac{\partial\phi}{\partial y}(x,y)\] 
 are integrable with respect to $\mu_\partial$ and $\mu$, respectively.
\end{lemma}
\begin{proof}
 This is true for $p=+\infty$ because in that case the supports of the measures involved are bounded and the integrands are continuous. For $p\in [1,+\infty)$ we have, from \eqref{eq:Lp},
 \begin{multline*}\int_{\partial\Omega\times Y}\|\phi(x,y)\mathbf n(x)\|\,d\mu(x,y)\leq \int_{\partial\Omega\times Y}c(1+\|y\|^p)\|\mathbf n(x)\|\,d\mu_\partial(x,y)\\=\int_{\partial\Omega\times Y}c(1+\|y\|^p)\,d\mu_\partial(x,y)<+\infty
 \end{multline*}
 and, by the H\"older inequality with $\frac1p+\frac1q=1$, so that $p/q=p-1$, and using \eqref{eq:Lp},
 \begin{align*} 
  \int_{\Omega\times Y\times Z}&\left\|\frac{\partial\phi}{\partial x}(x,y)+z^\top\frac{\partial \phi}{\partial y}(x,y)\right\|_1\,d\mu(x,y,z)\\
  &\leq   \int_{\Omega\times Y\times Z}\left\|\frac{\partial\phi}{\partial x}(x,y)\right\|+\sum_i\left|z_i^\top\frac{\partial \phi}{\partial y}(x,y)\right|\,d\mu(x,y,z)\\
    &\leq   \int_{\Omega\times Y\times Z}c(1+\left\|y\right\|^p)d\mu(x,y,z)\\
    &\qquad+\sum_i\left(\int_{\Omega\times Y\times Z}\left\|z_i\right\|^{p}d\mu(x,y,z)\right)^{\frac1p}\left(\int_{\Omega\times Y\times Z}\left\|\frac{\partial \phi}{\partial y}(x,y)\right\|^q\,d\mu(x,y,z)\right)^{\frac1q}\\
    &\leq   \int_{\Omega\times Y\times Z}c(1+\left\|y\right\|^p)d\mu(x,y,z)\\
    &\qquad+\sum_i\left(\int_{\Omega\times Y\times Z}\left\|z_i\right\|^{p}d\mu(x,y,z)\right)^{\frac1p}\left(\int_{\Omega\times Y\times Z} c^q(1+\left\|y\right\|^{p-1})^q\,d\mu(x,y,z)\right)^{\frac1q}\\
    &\leq   \int_{\Omega\times Y\times Z}c(1+\left\|y\right\|^p)d\mu(x,y,z)\\
    &\qquad+\sum_i\left(\int_{\Omega\times Y\times Z}\left\|z_i\right\|^{p}d\mu(x,y,z)\right)^{\frac1p}\\
    &\qquad\times\left(\left(\int_{\Omega\times Y\times Z} c^qd\mu(x,y,z)\right)^{\frac1q}+\left(\int_{\Omega\times Y\times Z} c^q\left\|y\right\|^{p-1})^q\,d\mu(x,y,z)\right)^{\frac1q}\right)\\
    &\leq   \int_{\Omega\times Y\times Z}c(1+\left\|y\right\|^p)d\mu(x,y,z)\\
    &\qquad+\sum_i\left(\int_{\Omega\times Y\times Z}\left\|z_i\right\|^{p}d\mu(x,y,z)\right)^{\frac1p}\\
    &\qquad\times
    \left(c\mu(\Omega\times Y\times Z)^{\frac{1}{q}}+c\left(\int_{\Omega\times Y\times Z} \left\|y\right\|^{p}\,d\mu(x,y,z)\right)^{\frac1q}\right)\\
    &<+\infty.\qedhere
 \end{align*}
\end{proof}

{\bf Acknowledgment.}
The work of  supported by the Barrande project 44702QH  and the M\v{S}MT \v{C}R project 8J20FR019. M.~Kru\v{z}\'{i}k was further supported by the  GA\v{C}R project 21-06569K. R. R\'ios-Zertuche gratefully acknowledges the support of ANR-3IA Artificial and Natural Intelligence Toulouse Institute (ANITI) and the hospitality of LAAS-CNRS. The research of M. Korda is partly supported by AI Interdisciplinary Institute ANITI funding, through the French ``Investing for the Future PIA3'' program under the Grant agreement n$^{\circ}$ANR-19-PI3A-0004. This research is also part of the programme DesCartes and is supported by the National Research Foundation, Prime Minister's Office, Singapore under its Campus for Research Excellence and Technological Enterprise (CREATE) programme.


\begin{thebibliography}{XX}
\bibitem{aubinfrankowska}
 J.P. Aubin, H. Frankowska. {Set-valued analysis.} Springer, 2009.
 
 \bibitem{Brown}
 W.F. Brown Jr. {Micromagnetics.} John Wiley \& Sons, 1963.
 
 \bibitem{aliprantisborder}
 D. Charalambos, B. Aliprantis. {Infinite dimensional analysis: a hitchhiker's guide.} Springer, 2013.
 
\bibitem{cbfg21}
A. Chernyavsky, J. J. Bramburger, G. Fantuzzi, D. Goluskin. Convex relaxations of integral variational problems: pointwise dual relaxation and sum-of-squares optimization. arXiv:2110.03079, 2021.

\bibitem{d04}
B. Dacorogna. Introduction to the calculus of variations. Imperial College Press, 2004.

\bibitem{desim}
  A. DeSimone.  Energy minimizers for large ferromagnetic bodies. {Arch. Rat. Mech. Anal.} 125:99--143, 1993. 

\bibitem{EvansGariepy} L. C. Evans, R. F. Gariepy. Measure theory and fine properties of functions. CRC Press, 1992.


\bibitem{ft22}
G. Fantuzzi, I. Tobasco. Sharpness and non-sharpness of occupation measure bounds for integral variational problems. arXiv:2207.13570, 2022.

\bibitem{gq09}
V. Gaitsgory, M. Quincampoix. Linear programming analysis of deterministic infinite horizon optimal
control problems with discounting. SIAM J. Control Optim. 48:2480--2512, 2009.

\bibitem{hk14}
D. Henrion, M. Korda. Convex computation of the region of attraction of polynomial control systems.
IEEE Trans. Autom. Control 59:297--312, 2014.

\bibitem{hkl20}
D. Henrion, M. Korda, J. B. Lasserre. The moment-SOS hierarchy - Lectures in probability, statistics, computational geometry, control and nonlinear PDEs. World Scientific, 2020.


\bibitem{hkw19}
D. Henrion, M. Kru\v z\'{\i}k, T. Weisser. Optimal control problems with oscillations, concentrations and discontinuities. Automatica 103:159--165,  2019.


\bibitem{jam-kin}
R. D. James, D. Kinderlehrer.  Frustration in ferromagnetic materials.
Continuum Mech. Thermodyn. 2:215--239, 1990.

\bibitem{kohn}
R. V. Kohn. The relaxation of a double-well energy. Continuum Mech. Thermodyn. 3:193--236, 1991.

\bibitem{khl21}
M. Korda, D. Henrion, J. B. Lasserre. Moments and convex optimization for analysis and control of nonlinear partial differential equations. In E. Tr\'elat, E. Zuazua (Eds). Numerical Control. Handbook of Numerical Analysis, Elsevier, 2021.

\bibitem{kr22}
M. Korda, R. R\'{\i}os-Zertuche.
The gap between a variational problem and its occupation measure relaxation. arXiv:2205.14132, 2022.

\bibitem{lhpt08}
J. B. Lasserre, D. Henrion, C. Prieur, E. Tr\'elat. Nonlinear optimal control via occupation measures and LMI relaxations. SIAM J. Control Optim. 47(4):1643-1666, 2008.

\bibitem{lukevnetukavesely}
 J. Luke{\v{s}},  I. Netuka, J. Vesel\'{y}.  Choquet's theory and the Dirichlet problem. Expositiones Mathematicae 20:229--254, 2002.

\bibitem{mwhl20}
S. Marx, T. Weisser, D. Henrion, J. B. Lasserre. A moment approach for entropy solutions to nonlinear hyperbolic PDEs.  Mathematical Control and Related Fields 10:113--140, 2020.

\bibitem{pedregal0}
 P. Pedregal.  Relaxation in ferromagnetism: the rigid case,  J. Nonlinear Sci. 4:105--125, 1994.
 
\bibitem{pedregalenv}
 P. Pedregal. {Parametrized measures and variational principles}, Birkh\"{a}user, 1997. 
 
 \bibitem{raoult}
A. Raoult. Quasiconvex envelopes in nonlinear elasticity. In J. Schröder, P. Neff (Eds).   Poly-, quasi- and rank-one convexity in applied mechanics. CISM International Centre for Mechanical Sciences, vol 516. Springer, 2010.

 \bibitem{roubicek}
T. Roub\'{\i}\v cek. Relaxation in optimization theory and variational calculus. W. de Gruyter, 2nd ed., 2020.

\bibitem{rudin}
W. Rudin. {Real and complex analysis.} McGraw-Hill, 1968.



\bibitem{v93}
R. Vinter. Convex duality and nonlinear optimal control. SIAM J. Control Optim. 31:518--538, 1993.

 \bibitem{y69}
L. C. Young. Lectures on the calculus of variations and optimal control theory. W. B.
Saunders Co., 1969.


 




 
\end{thebibliography}
\end{document}